\documentclass[11pt]{article}

\usepackage{graphicx} 
\usepackage{subcaption}
\usepackage{float} 
\usepackage{amssymb}
\usepackage{amsmath}
\usepackage{enumitem} 
\usepackage{mathtools}
\usepackage[thmmarks,amsmath]{ntheorem} 
\usepackage{lineno}
\usepackage{xcolor}
\definecolor{darkblue}{RGB}{0, 0, 100}
\definecolor{darkorange}{RGB}{200, 100, 0}

\usepackage[numbers]{natbib}  
\usepackage{bm} 
\usepackage{extarrows}
\usepackage{mathrsfs} 
\usepackage{algorithm}         
\usepackage{algorithmic}      
\usepackage{enumitem}         
\usepackage{algorithm}
\usepackage{algorithmic}
\usepackage{amsmath,amssymb}

\usepackage{lipsum}

\usepackage{booktabs} 
{
    \theoremstyle{nonumberplain}
    \theoremheaderfont{\bfseries}
    \theorembodyfont{\normalfont}
    \theoremsymbol{\mbox{$\Box$}}
    \newtheorem{proof}{Proof}
}
\makeatletter
\newcounter{mycounter}[section]

{
    \theoremstyle{plain}
    \theoremheaderfont{\bfseries}
    \theorembodyfont{\normalfont}
    \newtheorem{theorem}[mycounter]{Theorem}
    
    \newtheorem{lemma}[mycounter]{Lemma}

    \newtheorem{claim}[mycounter]{Claim}

    \newtheorem{conjecture}[mycounter]{Conjecture}
}

{
    \theoremheaderfont{\bfseries}
    \theorembodyfont{\normalfont}
    
}

\makeatother

\newcommand{\RNum}[1]{\uppercase\expandafter{\romannumeral #1\relax}}
\allowdisplaybreaks[4]

\graphicspath{{figure/}}                                 
\usepackage[a4paper]{geometry}                        
\title{Highly connected spanning oriented subdigraphs in generalizations of semicomplete digraphs \thanks{The author's work is supported by National Natural Science Foundation of China (No.12571373)}}
\author{Jia Zhou\textsuperscript{1}, J{\o}rgen Bang-Jensen\textsuperscript{2,3}\footnote{Corresponding author. E-mail address: jbj@imada.sdu.dk.} , Tong Zhou\textsuperscript{2}, Jin Yan\textsuperscript{2} %
\unskip\\[2mm]
\textsuperscript{1} School of Mathematics and Statistics, Ningxia University, Yinchuan 750021, China\\
\textsuperscript{2} School of Mathematics, Shandong University, Jinan 250100, China
\\
\textsuperscript{3} Department of Mathematics and Computer Science, University of\\ Southern Denmark, Odense DK-5230, Denmark}
\date{ }

\begin{document}
\maketitle

\begin{abstract}
Let $k$ be a positive integer. Jackson and Thomassen conjectured in 1989 that there exists an integer  function $f(k)$ such that every $f(k)$-strong digraph admits a spanning $k$-strong oriented subdigraph. They even conjectured that one can take $f(k)=2k$ [Ann. N. Y. Acad. Sci. 555 (1989) 402–412]. Already the existence of $f(2)$ is open for general digraphs. Thomassen proved that $f(2)=4$ for symmetric digraphs.  For general 
$k$, the existence of $f(k)$ was only known for {symmetric digraphs (every arc is in a 2-cycle),} locally semicomplete digraphs and quasi-transitive digraphs.  
In this paper, we prove the existence of $f(k)$ for two subclasses of the semicomplete multipartite digraphs, namely extended semicomplete digraphs and semicomplete split digraphs. We prove that every $(4k+1)$-strong extended semicomplete digraph contains a spanning $k$-strong oriented subdigraph and every $5k$-strong semicomplete split digraph contains a spanning $k$-strong oriented subdigraph. The first result implies that for the large class of digraphs which can be obtained from some semicomplete digraph $S$ on at least 3 vertices by substituting arbitrary digraphs for each vertex of $S$ we also have $f(k)\leq 4k+1$.
\end{abstract}
\maketitle

\vspace{1ex}
\vspace{1ex}


\section{Introduction}

An {\bf orientation} of an undirected graph $G=(V,E)$ is an assignment of one of the two possible orientations $x\to y$ or $y\to x$ 
to each edge $xy$ of $E$. A classical result due to Robbins \cite{robbins1939} says that a connected graph has a strongly connected orientation if and only if $G$ is 2-edge-connected. 
Nash-Williams \cite{nashwilliamsCJM12} generalized Robbins' theorem to general arc-connectivities. A digraph $D=(V,A)$ is {\bf $\boldsymbol{k}$-arc-strong} if it remains strongly connected after the removal of any subset of $k-1$ arcs. 

\begin{theorem}[Nash-Williams Orientation Theorem]\cite{nashwilliamsCJM12}\label{thm:nwthm}
A graph $G=(V,E)$ has a $k$-arc-strong orientation if and only if $G$ is $2k$-edge-connected.
\end{theorem}

As every $k$-arc-strong digraph must have at least $k$ arcs in both directions for every 2-partition $X,V\setminus X$ of its vertex set, the requirement that $G$ is  $2k$-edge-connected is clearly necessary in Theorem \ref{thm:nwthm}.

A digraph $D=(V,A)$ is {\bf $\boldsymbol{k}$-strong} if it has at least $k+1$ vertices and remains strong after the deletion of any subset of $k-1$ vertices. For orientations maintaining high vertex-connectivity 
there is no obvious necessary condition other than $2k$-edge-connected (if some cut has less than $2k$ edges, then we cannot even obtain a $k$-arc-strong orientation).
Thomassen \cite{thomassen1989} made the following conjecture

\begin{conjecture}\cite{thomassen1989}
\label{conj:kstrongorG}
There exists an integer function $g(k)$ such that every $g(k)$-connected graph has a $k$-strong orientation.
\end{conjecture}

The graph obtained from two copies of $K_{2k-1}$ by adding the edges of a matching between the two copies is $(2k-1)$-connected but does not even have a $k$-arc-strong orientation so $g(k)\geq 2k$ must hold. Conjecture \ref{conj:kstrongorG} is open for every $k\geq 3$.

Jord\'an \cite{jordanJCTB95} was the first to prove the existence of $g(2)$. He proved that $g(2)\leq 18$.
Later, confirming the special case of 2-strong connectivity of a conjecture of Frank \cite{frank1995},  Thomassen proved the following which implies that $g(2)=4$.

\begin{theorem}\cite{thomassenJCTB110}
\label{thm:2-strongor}
A graph $G=(V,E)$ has a 2-strong orientation if and only if $G-v$ is 2-edge-connected for every $v\in V$.
\end{theorem}

{Conjecture \ref{conj:kstrongorG} was proved very recently by Garamv\"olgyi, Jordan, Kir\'aly and Vill\'anyi who proved that $g(k)\leq 320k^2$.

\begin{theorem}\cite{garamvolgyiFM13}
  \label{thm:kstrongorG}
  For every integer $k\geq 1$ every $320k^2$-connected graph has a $k$-strong orientation.
  \end{theorem}}
Every graph $G=(V,E)$ has a corresponding digraph 
$\stackrel{\leftrightarrow}{G}$, called the {\bf complete biorientation of $G$}, which is obtained by replacing every edge $xy$ of $E$ by the directed 
2-cycle $xyx$. A digraph $D$ is {\bf symmetric} if $D=\stackrel{\leftrightarrow}{G}$ for some graph $G$. Clearly $\stackrel{\leftrightarrow}{G}$ is $r$-strong if and only if $G$ is $r$-connected and $G$ has a $k$-strong orientation if and only if we can delete one arc of every 2-cycle in $\stackrel{\leftrightarrow}{G}$ to obtain a $k$-strong oriented graph (an orientation of $G$). This motivates the definition of an orientation of a general digraph $D$:
An {\bf orientation} of a digraph $D=(V,A)$ is any oriented spanning subdigraph of $D$ which we can obtain by deleting precisely one arc from every 2-cycle of $D$.
{Theorem \ref{thm:kstrongorG} is clearly equivalent to the following.

\begin{theorem}\label{thm:orientsymmetric}
    For every integer $k\geq 1$ every $320k^2$-strong symmetric digraph has a $k$-strong orientation.
\end{theorem}}

 Theorem \ref{thm:nwthm}, which  implies that a $2k$-arc-strong symmetric digraph  has a $k$-arc-strong orientation, was generalized to orientations 
of  arbitrary digraphs by  Jackson \cite{jackson1988} who proved
the following result. Theorem \ref{thm:mixedkarcstrong} can also be deduced  from a more general orientation theorem of Frank  \cite{frankJCTB28}. The case $k=1$ was already proved in \cite{boesch1980}.

\begin{theorem}\cite{frankJCTB28,jackson1988}
\label{thm:mixedkarcstrong}
    Every $2k$-arc-strong digraph has a spanning $k$-arc-strong orientation
\end{theorem}

Motivated by this result and conversations with Jackson,
 Thomassen posed the following conjecture in \cite{thomassen1989}.

\begin{conjecture}\label{conj1}
(Jackson and Thomassen \cite{thomassen1989}) Every $2k$-strong digraph contains a spanning $k$-strong oriented graph.
\end{conjecture}

Since a digraph is strong if and only if it is 1-arc-strong, it follows from Theorem \ref{thm:mixedkarcstrong} that the conjecture holds for $k=1$. Already for $k=2$  the conjecture is open for general digraphs, while  Theorem \ref{thm:2-strongor} implies that Conjecture \ref{conj1} holds for $k=2$ when $D$ is a symmetric digraph. 

{In fact, the following natural weakening of Conjecture \ref{conj1}  is open for general digraphs already for $k=2$. By Theorem \ref{thm:orientsymmetric} the conjecture holds for all $k$ in the case of symmetric digraphs.}

\begin{conjecture}\label{conj2}
There exists a function $f(k)$ so that every $f(k)$-strong digraph on at least $2k+1$ vertices contains a spanning $k$-strong oriented graph.
\end{conjecture}

Clearly $f(k)\geq g(k)\geq 2k$ must hold for general digraphs. A digraph is {\bf semicomplete} if it has no pair of non-adjacent vertices. A {\bf tournament} is an orientation of a complete graph, i.e. a semicomplete digraph with no directed 2-cycles.
 Bang-Jensen and Jordán \cite{bangDM310}  proved  that every 3-strong semicomplete digraph on at least 5 vertices contains a spanning 2-strong tournament. Based on this, they conjectured that $f(k)=2k-1$ for semicomplete digraphs.

 \begin{conjecture}\label{conj:SDcase}\cite{bangDM310}
     Every $(2k-1)$-strong semicomplete digraph on at least $2k+1$ vertices contains a spanning $k$ strong tournament.
 \end{conjecture}

 They also constructed an infinite family of $(2k-2)$-strong semicomplete digraphs with no spanning $k$-strong subtournament. Conjecture \ref{conj:SDcase} was verified for $k=3$ in \cite{wangDM347}.

A digraph is {\bf locally semicomplete} if the out-neighbours and the in-neighbours of every vertex induce a semicomplete digraph.
 Bang-Jensen (see Theorem 4.5 in \cite{bangJGT14}) proved that every $5k$-strong locally semicomplete digraph contains a spanning $k$-strong tournament. Extending the proof technique, Guo was able to improve this to $3k-2$.

\begin{theorem}\cite{guoDAM79}
\label{Thm:LSDcase} 
For every positive integer $k$, every $(3k - 2)$-strong locally semicomplete digraph contains a spanning $k$-strong local tournament.
\end{theorem}

A digraph is {\bf quasi-transitive} is it contains no induced directed path of length 2.
From Theorem \ref{Thm:LSDcase} and the structure theorem for quasi-transitive digraphs (see Theorem 2.7.5 in \cite{book}) one can derive the following.

\begin{theorem}\cite[Corollary 11.10.6]{book}
    For every positive integer $k$, every $(3k - 2)$-strong quasi-transitive digraph contains a spanning $k$-strong extended tournament.
\end{theorem}

A digraph $D$ is {\bf composed} from its induced subdigraphs $H,H_1,H_2,\dots{},H_h$ if we have $V(D)=\bigcup_{j\in [h]} V(H_j)$, $D$ contains all arcs of each $H_i$ and we can label $V(H)=\{v_1,v_2,\dots{},v_h\}$ so that $v_i\in V(H_i)$ and the adjacency between each pair of vertices $x\in V(H_i),y\in V(H_j)$ is the same as between the vertices $v_i,v_j$ in $H$. We write $D=H[H_1,H_2,\dots{},H_h]$ to denote this composition and say that $D$ is {\bf decomposable}.
  We call $D$ an {\bf extension of } $H$ is $D=H[H_1,\dots{},H_h]$ where each $H_i$ is an arcless digraph.
   A {\bf semicomplete composition} is a composition $D=H[H_1,H_2,\dots{},H_h]$ where $H$ is a semicomplete digraph. An {\bf extended semicomplete digraph} is an extension of a semicomplete digraph. In particular, every extended semicomplete digraph is a semicomplete composition.
  
  A digraph is \textbf{semicomplete multipartite} if it is obtained from a complete multipartite graph by replacing every edge by an arc or a pair of opposite arcs.  
  In particular, every (extended) semicomplete digraph is a semicomplete multipartite digraph. Finally, a {\bf semicomplete split digraph} is a semicomplete multipartite digraph where only one of the partite (independent) sets may have size more than one.

 In this paper, we show the existence of $f(k)$ for all $k$ when $D$ belongs to  the three digraph classes, each  of which are generalizations of semicomplete digraphs, namely  extended semicomplete digraphs, semicomplete split digraphs and semicomplete compositions for semicomplete digraphs on at least 3 vertices. This provides additional support for Conjecture \ref{conj2}. 

Extended semicomplete digraphs constitute a very important subclass of the semicomplete multipartite digraphs. They play a crucial role in both structural and algorithmic properties of semicomplete compositions, in particular in  the solution of the hamiltonian path and cycle problems for  quasi-transitive digraphs. For an illustration of this, see  e.g. Section 6.7 in \cite{book}.

The main results of this paper are as follows.
 
\begin{theorem}\label{extended}
    Let $k$ be a positive integer. Every $(4k+1)$-strong extended semicomplete digraph contains a spanning $k$-strong oriented subdigraph.
\end{theorem}

Using Theorem \ref{extended} and the fact that, by Lemma \ref{scd-com} below, the connectivity of a semicomplete composition $D=H[H_1,\dots, H_h]$, with $h\geq 3$ is the same as  the connectivity of the spanning extended semicomplete subdigraph $D_0$ which one obtains by deleting all arcs inside each $H_i$ we can derive the following.

\begin{theorem}\label{main1}
 For every integer $k \geq 1$, every $(4k+1)$-strong semicomplete composition  $D = H[S_1, \dots, S_h]$ {such that $h\geq 3$ or} \(|V(D) \setminus S_i| \geq 2k\) for all $S_i$ contains a spanning $k$-strong oriented subgraph.
\end{theorem}

Note that in the case $h=2$ of a semicomplete composition $D=H[H_1,H_2]$ where $H_2$ has $r<k$ vertices, the connectivity of $D$ is just $r$ larger than the connectivity of the arbitrary digraph $H_1$. Thus, if we could prove the existence of $f(k)$ for  every $k$ in this subclass of digraphs, we would have proved Conjecture \ref{conj2}.

Our final result is for semicomplete split digraphs.
\begin{theorem}\label{thm22}
     Let $k$ be a positive integer. Every $5k$-strong semicomplete split digraph $D=(V_1,V_2,A)$ contains a spanning \( k \)-strong oriented semicomplete split digraph.
\end{theorem}

The rest of the paper is organized as follows. In section 2, we collect and establish some lemmas. We then prove Theorems \ref{extended}-\ref{thm22} in sections 3, 4 and 5, respectively.

\section{Notation and preliminaries}

Notation not introduced here is consistent with \cite{book}. Let $\mathbb{N}$ be the set of natural numbers. For an integer $i$, we use the notation  $\boldsymbol{[i]} = \{1, \ldots, i\}$. The digraphs considered in this paper are always simple, i.e., without loops and multiple arcs. Let $D$ be a digraph with vertex set $V(D)$ and arc set $A(D)$. We use $\boldsymbol{|D|}$ to represent the number of vertices in $D$. For two vertices $x, y \in V(D)$, we denote the arc from $x$ to $y$ as $xy$.
An arc $xy$ is a {\bf single arc} of $D$ if $D$ does not contain the opposite arc $yx$, i.e. $x$ and $y$ do not induce a 2-cycle. Suppose \( xy \) is a single arc of \( D \). We denote by \( \boldsymbol{D_{\{\stackrel{\leftarrow}{xy}\}}} \) the digraph obtained from \( D \) by deleting the arc \( xy \) and adding the new arc \( yx\). For a vertex $x$ of $D$, we define $\boldsymbol{N^+_D(x)} = \{y \mid xy \in A(D)\}$ (resp. $\boldsymbol{N^-_D(x)} = \{y \mid yx \in A(D)\}$) as the \textbf{out-neighborhood} (resp. \textbf{in-neighborhood}) of $x$, and $\boldsymbol{d^+_D(x)} = |N^+_D(x)|$ (resp. $\boldsymbol{d^-_D(x)} = |N^-_D(x)|$) as the \textbf{out-degree} (resp. \textbf{in-degree}) of $x$.  
For a set $X \subseteq V(D)$, the subgraph of $D$ induced by $X$ is denoted by $\boldsymbol{D[X] }$, and the digraph obtained from $D$ by deleting $X$ and all arcs incident with $X$ is denoted by $\boldsymbol{D \setminus X}$. For disjoint sets \(X,Y\subseteq V(D)\):
\begin{itemize}
\item \(\boldsymbol{X \to Y}\) means every vertex in \(X\) dominates every vertex in \(Y\);
\item \(\boldsymbol{X \leftrightarrow Y}\) means that $X \to Y$ and $Y\to X$;
\item $\boldsymbol{X \mapsto Y}$ means that $X \to Y$, while no arc exists in the reverse direction from $Y$ to $X$.
\end{itemize}


Suppose below that $P = x_1x_2 \cdots x_t$ is a directed path of $D$. We say that  $x_1$ (resp. $x_t$) is  the \textbf{initial} (resp. \textbf{terminal}) vertex of $P$ and that $P$ is an $(x_1,x_t)$-path. The \textbf{length} of $P$ is the number of arcs, and we denote a path of length $l$ as an \textbf{$\boldsymbol{l}$-path}.  Suppose that \(Q = x_t x_{t + 1}\cdots x_m\) is a path in \(D\) that is internally disjoint from $P$, and we denote the \textbf{concatenation} of \(P\) and \(Q\) by \(\boldsymbol{P \circ Q}\), which is the directed path $x_1x_2\cdots x_tx_{t+1}\cdots x_m$. 
A directed multigraph \(D=(V,A)\) is \textbf{minimally \(k\)-strong} if \(D\) is \(k\)-strong, but for every arc \(e\in A\), \(D-e\) is not \(k\)-strong. 

\subsection{Preliminary Lemmas}
\leavevmode\par 
The following {important} lemma implies that almost all minimally $k$-strong decomposable digraphs are subdigraphs of extensions of digraphs. We use the relationship between extended semicomplete digraphs and semicomplete compositions to prove Theorem \ref{extended}.

\begin{lemma}\label{scd-com}\cite{jbj1999}
 Let $D = H[S_1, \dots, S_h]$ be a composition, where $H$ is a strong digraph on $h \geq 2$ vertices. Let $D_0 = H[S_1^0, \dots, S_h^0]$ be the digraph obtained from $D$ by deleting every arc which lies inside $S_i$, for all $i \in [h]$. If $H$ has at least three vertices, then $D$ is $k$-strong if and only if $D_0$ is $k$-strong.
\end{lemma}

Menger's theorem is one of the cornerstone results in graph connectivity, and we use the following version of the theorem.
\begin{theorem}\label{menger}
\cite{menger} Let \( D \) be a \( k \)-strong digraph. Let \( x_1, x_2, \ldots, x_r, y_1, y_2, \ldots, y_s \) be distinct vertices of \( D \), and let \( a_1, a_2, \ldots, a_r \), \( b_1, b_2, \ldots, b_s \) be positive integers such that
\[
\sum_{i=1}^r a_i = \sum_{j=1}^s b_j = k.
\]
Then $D$ contains $k$ internally disjoint paths $P_1, P_2, \ldots, P_k$ with the property that precisely $a_i$ \textup{(}$b_j$\textup{)} of these start at $x_i$ \textup{(}end at $y_j$\textup{)}.
\end{theorem}


The following theorem established by Volkmann facilitates our proof of the case $k=1$ in Theorem \ref{extended}.
\begin{theorem}\label{SPD}\cite{62}
    Every strong semicomplete $c$-partite digraph with $c \geq 3$ contains a spanning strong oriented subdigraph.
\end{theorem}

The following lemma, which was also used in \cite{bangJGT14} to prove that every $5k$-strong semicomplete digraph has a spanning $k$-strong tournament, plays an essential role in our proofs.

\begin{lemma}\label{keylemma}\cite{book}
Let \( k \) be positive integer, and let \( D \) be a \( (k-1) \)-strong digraph. Suppose that \( xy \) is a single arc of $D$ and let 
$D' = D_{\stackrel{\leftarrow}{xy}}$. If there exist at least \( k \) internally disjoint \( (x,y) \)-paths in \( D' \), then \( D' \) is \( (k-1) \)-strong. Furthermore, if \( D' \) is not \( k \)-strong, then every minimum separating set \( S' \) of \( D' \) also separates \( D \).
\end{lemma}

\section{Proof of Theorem \ref{extended}}

We repeat the statement of the theorem here.

\noindent{}{\bf Theorem \ref{extended}} Let $k$ be a positive integer. Every $(4k+1)$-strong extended semicomplete digraph contains a spanning $k$-strong oriented subdigraph.

{Before presenting the detailed proof, a brief outline is provided below. This theorem is proved by induction on the parameter $k$. Arc orientations are assigned following designated rules (see rules (a)--(c)) to construct a spanning semicomplete multipartite subdigraph $D'$ of $D$, and the strong connectivity of $D'$ (see Claim \ref{claim1}) is verified to establish the base case of induction. For the inductive step on $k$, the proof proceeds by contradiction. An optimal spanning $(k-1)$-strong oriented subdigraph $T'$ in $D'$ is selected, with $S$ denoting its minimum separating set. The vertex set $V(T')$ is partitioned into three disjoint subsets $S$, $U$, and $W$, where no arcs directed from $U$ to $W$ in $T'$. Combined with the above structural and strong connectivity, both $U\setminus S'$ and $W\setminus S'$ are verified to be nonempty (see Claim \ref{non-empty}), where $S' \subset S$. The subsequent argument is divided into two cases based on the distribution features of some partite sets. In each case, a vertex pair $(x,y)$ with $k+1$ internally disjoint $(x,y)$-paths in $T'$ exists, where $x$ and $y$ form a $2$-cycle in $D'$ (see Claim \ref{case2.1} and Case 2). However, such vertex pairs are excluded by the optimality of $T'$ (see Claim \ref{contra}). The resulting contradiction completes the entire inductive proof.}

\begin{proof}{\textbf{of Theorem \ref{extended}.}}    
Let $D = H[S_1, S_2, \dots, S_h]$ be an $(4k+1)$-strong extended semicomplete digraph. Suppose first that $h=2$. By definition of extended semicomplete digraphs, every vertex in $S_1$ forms a 2-cycle with every vertex of $S_2$. As $D$ is $(4k+1)$-strong we have $|S_i|\geq 4k+1$ for $i\in [2]$. Let $S_i^1$ consist of and arbitrary subset of $S_i$ of size $2k$ for $i\in [2]$ and let $S_i^2=S_i\setminus S_i^1$. Then $D$ contains the arcs
$S_1^1\to S_2^1\to S_1^2\to S_2^2\to S_1^1$ and these arcs induce a spanning $2k$-strong oriented subdigraph of $D$.

In what follows we assume $h\geq 3$. Now our aim is to construct a spanning semicomplete $h$-partite subdigraph $D'\subseteq D$. Before that, we introduce necessary notations and partitions. A partition set $S_i$ is called a \textbf{large set} if $|S_i|\geq 2$, and a \textbf{singleton set} if $|S_i|=1$. For each $S_i$, partition $S_i$ into three disjoint subsets $S_i^1,S_i^2,S_i^3$ satisfying
\[
\begin{cases}
|S_i^1| = |S_i^2| = \dfrac{|S_i|}{2},\ S_i^3=\emptyset, & |S_i|\text{ even};\\[4pt]
|S_i^1| = |S_i^2| = \left\lfloor \dfrac{|S_i|}{2} \right\rfloor,\ |S_i^3|=1, & |S_i|\text{ odd}.
\end{cases}
\]
In particular, if $|S_i|=1$, then $S_i^1=S_i^2=\emptyset$ and $S_i^3=S_i$.
For each part $S_i$, define the \textbf{external vertex set}
\[
{EX}(S_i)=\bigcup_{t\in [h]\setminus \{i\}}S_t^3.
\]
For each large set $S_i$, let $C(S_i)\subseteq {EX}(S_i)$ be the set of vertices which form a 2-cycle with every vertex in $S_i$ in $D$.

We now construct a spanning semicomplete $h$-partite subdigraph $D'$ of $D=H[S_1,S_2,\dots,S_h]$. The arc set of $D'$ is defined according to the ordering of partition classes in $H[S_1,S_2,\dots,S_h]$, subject to the rules below.
The relevant structural illustrations are presented in Figure \ref{fig1}.
\begin{enumerate}[label=(\alph*)]
\item Let $S_i,S_j$ be distinct large sets with $i<j$ and $S_i\leftrightarrow S_j$ in $D$. We orient the arcs between $S_i^1\cup S_i^2$ and $S_j^1\cup S_j^2$ in $D'$ as $S_i^1\mapsto S_j^1$, $S_j^1\mapsto S_i^2$, $S_i^2\mapsto S_j^2$, $S_j^2\mapsto S_i^1$.
\item For any large set $S_i$:
    \begin{itemize}
    \item if $|C(S_i)|=1$, then $S_i\leftrightarrow v_1$ holds in $D'$;
    \item if $|C(S_i)|\geq 2$, we keep $S_i^3\leftrightarrow C(S_i)$ in $D'$. Partition $C(S_i)$ into two subsets of sizes $\lfloor |C(S_i)|/2\rfloor$ and $\lceil |C(S_i)|/2\rceil$. For each vertex $v$ in the former subset, set $S_i^1\mapsto v\mapsto S_i^2$; for each vertex $v$ in the latter subset, set $S_i^2\mapsto v\mapsto S_i^1$.
    \end{itemize}
\item All $2$-cycles connecting two singleton sets in $D$ are preserved in $D'$. All single arcs between distinct partition classes in $D$ are fully retained in $D'$.
\end{enumerate}

The above construction provides $D'$ with several crucial structural properties, which are summarized as follows. 
\begin{itemize}
    \item[(P1)] All $2$-cycles in $D\left[\bigcup_{t}S_t^3\right]$ and all single arcs belonging to $A(D)$ are preserved in $D'$.
    \item[(P2)] For each large set $S_j$, define $\mathcal{I}_j\subseteq \{1,\ldots,h\}$ to be the index set of all large sets that form $2$-cycles with $S_j$ in $D$. Then for any vertex $v\in S_j^1\cup S_j^2$ and any subset $I\subseteq \mathcal{I}_j$, the number of single out-arcs of $v$ in $D'$ is at least
    \[
    \frac{1}{2}d^{+}_{D\left[\left(\bigcup_{t\in \mathcal{I}_j}V(S_t)\right)\cup C(S_j)\right]}(v)-1.
    \](This lower bound follows from the bipartite orientation rule between large sets and external vertices, which preserves at least half of the original out-degree up to an error of at most one vertex.)
    \item[(P3)] If a large set $S_i$ contains at least one vertex $s_i$ lying on a $2$-cycle in $D'$, then either $S_i^3\neq\emptyset$ or $|C(S_i)|=1$.
\end{itemize}

\begin{figure}[htbp]
    \centering
\includegraphics[width=0.75\textwidth]{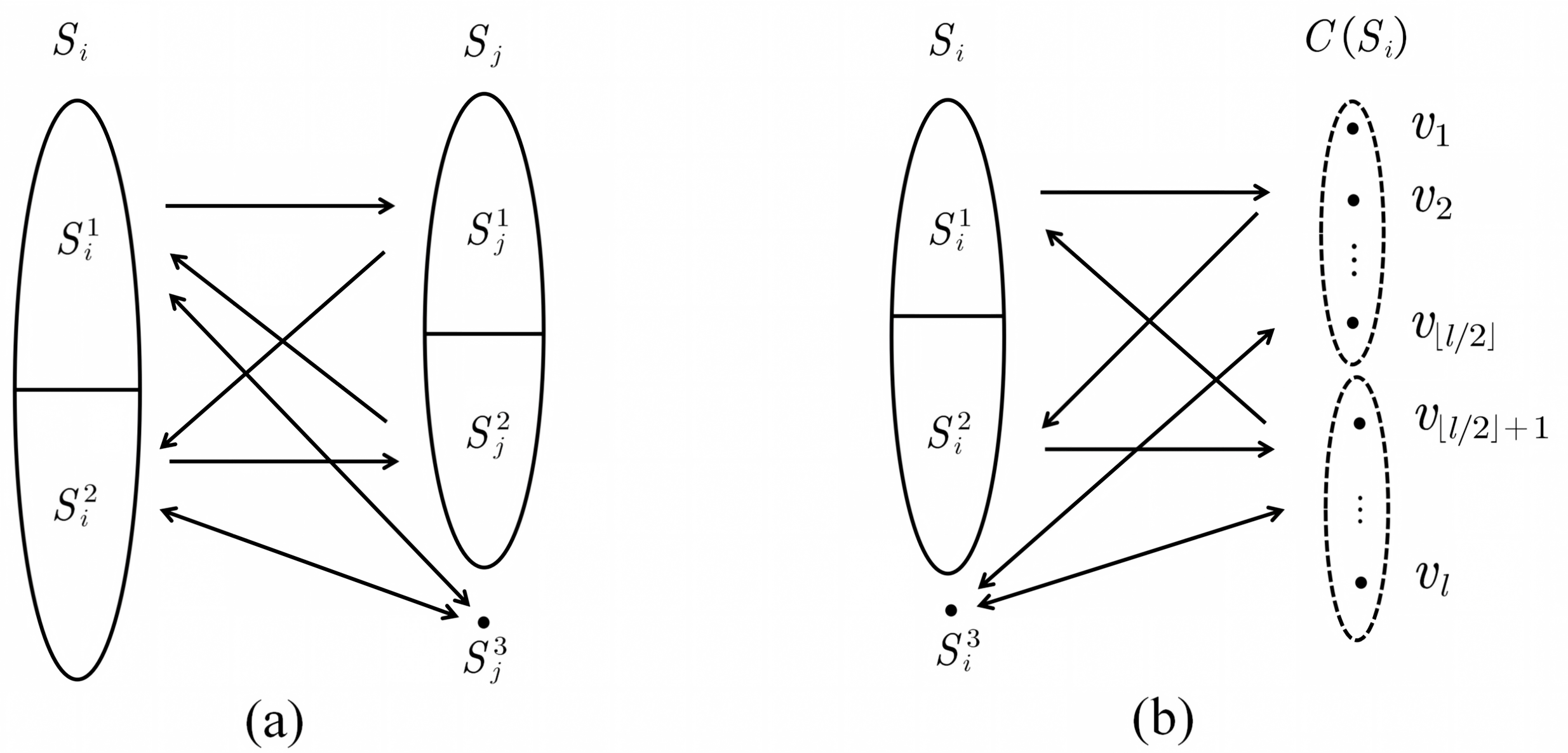}
    \caption{Construction of $D'$. (a) Arcs between two large sets. (b) Arcs between a large set and external vertices. A bidirectional arrow between $A$ and $B$ denotes the possible presence of arcs from $A$ to $B$ and from $B$ to $A$.}
    \label{fig1}
\end{figure}

 To prove by induction that $D'$ contains a spanning $k$-strong oriented subdigraph, we firstly prove that $D'$ is strong.
\begin{claim}\label{claim1}
$D'$ is strong.
\end{claim}
\begin{proof}
Suppose, to the contrary, that $D'$ is not strong. Then there exist two vertices $u,v \in V(D')$ such that there is no $(u,v)$-path in $D'$. Let $P := p_1 p_2 \cdots p_{l-1} p_l$ be a shortest $(u,v)$-path in $D$, such path exists since $D$ is strong. Assume that  $p_i \in S_{j_i}$ for each $i\in [l]$). If $p_i p_{i+1} \in A(D')$ for all $i\in [l-1]$, then $P$ is a $(u,v)$-path in $D'$, a contradiction. Thus there exists a smallest index $i$ such that $p_i p_{i+1} \notin A(D')$. By the construction of $D'$ ({maintaing all single arcs of $D$ in $D'$}), this can happen only if $S_{j_i}\leftrightarrow S_{j_{i+1}}$ in $D$. And it is impossible that both $S_{j_i}$ and $S_{j_{i+1}}$ are singleton sets by (c). Therefore, we distinguish cases according to the types of $S_{j_i}$ and $S_{j_{i+1}}$.

 \textbf{Case 1: Both $S_{j_i}$ and $S_{j_{i+1}}$ are large sets.} From rule (a), $D'$ contains an extended 4-cycle,  either $S_{j_i}^1\mapsto  S_{j_{i+1}}^1\mapsto  S_{j_i}^2\mapsto  S_{j_{i+1}}^2\mapsto   S_{j_i}^1$ or  $S_{j_{i+1}}^1\mapsto   S^1_{j_i}\mapsto  S_{j_{i+1}}^2\mapsto   S^2_{j_{i}}\mapsto   S^1_{j_{i+1}}$. It follows from the rule (b) that $S^3_{j_i}\leftrightarrow S^3_{j_{i+1}}$, $S^3_{j_i}$ has an out-neighbour in 
    $S^1_{j_{i+1}}\cup S^2_{j_{i+1}}$, and $S^3_{j_{i+1}}$ has an in-neighbour in $S^1_{j_{i}}\cup S^2_{j_{i}} $ (if  the sets $S^3_{j_i},S^3_{j_{i+1}}$ are not empty) in $D'$. Hence, one can easily find a $(p_i,p_{i+1})$-path inside $D'[S_{j_i} \cup S_{j_{i+1}}]$.

\textbf{Case 2: Exactly one of $S_{j_i}$ and $S_{j_{i+1}}$ is a large set.} 
Without loss of generality, let $S_{j_{i+1}}$ be a large set and $S_{j_i}=\{p_i\}$ be a singleton set, which implies $p_i\in EX(S_{j_{i+1}})$. Since $p_ip_{i+1}\in A(D)\setminus A(D')$, rule (b) ensures $p_{i+1}p_i\in A(D')$ and $p_{i+1}\notin S^3_{j_{i+1}}$ (assume $p_{i+1}\in S^2_{j_{i+1}}$). This means $|C(S_{j_{i+1}})|\geq 2$. Further by (b), there exist $v_1\in C(S_{j_{i+1}})\setminus \{p_i\}$ such that $v_1\mapsto  S^2_{j_{i+1}}\mapsto p_i\mapsto  S^1_{j_{i+1}}\mapsto v_1$ in $D'$. From this it is easy to see that $D'$ contains a $(p_i,p_{i+1})$-path. The symmetric case can be verified in the same manner.

There always exists a  $(p_i,p_{i+1})$-path in $D'$. As there exists a $(u,p_i)$-path in $D'$, combining these two paths produces a $(u,p_{i+1})$-path within $D'$. Iterating this procedure over all successive vertices of $P$ finally yields a $(u,v)$-path in $D'$, which leads to a contradiction.
\end{proof}

From now on, we prove that $D'$ contains a $k$-strong spanning oriented subdigraph $T$. {The proof below adopts similar ideas from \cite{bangJGT14} and \cite{guoDAM79}, but it involves far more sophisticated techniques owing to the structure of extended semicomplete digraphs.} 
Induction on $k$. The base case $k=1$ holds directly by Theorem \ref{SPD}, since $D'$ is strong and $h\geq 3$. Suppose $k \ge 2$ and the assertion is valid for $k-1$. We assume for contradiction that $D'$ contains no spanning $k$-strong oriented subdigraph. By the induction hypothesis, $D'$ has at least one spanning $(k-1)$-strong oriented subdigraph. We select such a subdigraph $T'$ satisfying the following lexicographic optimal conditions:
\begin{enumerate}[label=(\roman*), leftmargin=*]
    \item[(i)]  $A(T')\subseteq  A(D')$, and any pair of adjacent vertices in $D'$ are also adjacent in $T'$; 
    \item[(ii)] The number of minimum separating sets of $T'$ is as small as possible;
    \item[(iii)] Among all such subdigraphs, the number of strong components of $T'-S$ is minimized, where $S$ denotes an arbitrary minimum separating set of $T'$.
\end{enumerate}

Let $T'_1, T'_2, \dots, T'_p$ be the strong components of $T' - S$ arranged in acyclic topological order. We may assume $|V(T'_1)|\ge |V(T'_p)|$; otherwise we consider the converse digraph of $D'$ by reversing all arc directions. Define $W := \bigcup_{j=1}^{p-1} V(T'_j)$ and $U := V(T'_p)$.

We classify the large sets of {$D$} as follows. A large set $S_i$ is defined to be \textbf{bad} if $S_i^1 \subseteq S $ or $ S_i^2 \subseteq S $. All other large sets are defined to be \textbf{good}, which satisfy $S_i^1 \setminus S \neq \emptyset$ and $S_i^2 \setminus S \neq \emptyset$. Also, let $I_B := \{i\in [h] \mid S_i  \text{ is a bad large set}\}$ be the index set of bad large sets. Define the vertex set 
\begin{equation}\label{S'def}
S' := \bigcup_{i \in I_B} \left( (S_i^1 \cup S_i^2) \setminus S \right)
\end{equation}
and we have the cardinality relation
\begin{equation}\label{s'}
    | \bigcup_{i \in I_B}  (S_i^1 \cup S_i^2) |= |S\cap \bigcup_{i \in I_B}  (S_i^1 \cup S_i^2)|+|S'|\leq 2|S\cap \bigcup_{i \in I_B}  (S_i^1 \cup S_i^2)|.
\end{equation}
Since $|S_i^1| = |S_i^2|$ holds for every large set $S_i$, the contribution of each bad set to $S$ is at least its contribution to $S'$. Consequently, 
\begin{equation}\label{eq1}
    |S'| \le \sum_{i \in I_B} |S_i \cap S| \le |S| = k-1.
\end{equation} 
Combined with $|V(D)| \geq 4k+2>|S|+|S'|$, we deduce that at least one of $|U \setminus S'|\neq 0$ and $|W \setminus S'|\neq 0$ holds. Under the assumption $|W|\geq |U|$, if $W\setminus S'=\emptyset$, then
\[
|V|= |U \cup W \cup S| \leq 2|W| + |S| \leq 2(k-1) + (k-1) = 3k - 3,
\]
contradicting that \( D \) is \( (4k+1) \)-strong.
Hence \( W \setminus S' \neq \emptyset\).

\begin{claim}\label{non-empty}
  {\( U \setminus S'\neq\emptyset \)}
\end{claim}
\begin{proof}  
Assume $U\subseteq S'$ for contradiction. There exists a bad large set $S_i$ such that $S_i\cap U\neq\emptyset$. Fix a vertex $x\in S_i\cap U$, it is obvious that $d_{D[U]}^+(x) \leq |U| - 1$. We divide the out-neighbors of $x$ in $W$ within $D$ into three disjoint parts: $N_D^+(x)\cap W = A\cup B\cup G,$ where $A = N_{{W}}^+(x)\cap S'$, $B = N_{{W}}^+(x)\cap C(S_i)$, and $G$ consists of all remaining out-neighbors belonging to good large sets.

    The relation $U \subseteq S'$ implies $|A| \leq |S'| - |U|$. Let $l=|B|$. Every vertex in $B$ forms a $2$-cycle with $x$ in $D$. Based on the acyclic topological order of components of $T'$, there are no arcs directed from $U$ to $W$ in $T'$, so all edges between $x$ and vertices in $B$ are oriented towards $x$ in $T'$. From construction rule (b), there exist at least $l-1$ vertices in $C(S_i)$ satisfying $x\mapsto v$ in $D'$, hence these arcs are also contained in $T'$. To preserve the acyclic order, these $l - 1$ vertices cannot lie in $W$. Furthermore, they cannot lie in $U$ because \( U \subseteq S' \) and $C(S_i)\cap S'=\emptyset$. So they are all contained in $S$. This yields $l-1+|S'|\leq |S|$, namely $l\leq |S|-|S'|+1$.

We further estimate the out-degree of $x$ in $D$ excluding those in $G$, 
    \begin{align*}
        d_D^+(x) - |G| &\leq |S| + |A| + d_{D[U]}^+(x) + l \\
        &\leq |S| + (|S'| - |U|) + (|U| - 1) + (|S| - |S'| + 1) = 2|S| = 2k-2.
    \end{align*}
Since $D$ is $(4k+1)$-strong, $d_D^+(x)$ is sufficiently large, which forces $G\neq\emptyset$. Accordingly, there exists a good large set $S_j$ satisfying $S_i\leftrightarrow S_j$ in $D$ and $S_j\cap W\neq\emptyset$.  By rule (a), either \( x \mapsto S_j^1 \) or \( x \mapsto S_j^2 \) in \( D' \). Assume \( x \mapsto S_j^1 \), then the acyclic ordering of $T'$ forces \( S_j^1 \setminus S \subseteq U \) in $T'$, which contradicts $U\subseteq S'$. This completes the proof of Claim \ref{non-empty}.
 \end{proof}

Now $|U\setminus S'|\neq 0$ and $|W\setminus S'|\neq 0$. Since \( D \) is \( (4k+1) \)-strong  and $|S\cup S'|\leq 2k-2$, {Menger's theorem guarantees that} there exist at least \( 2k+3 \) pairwise distinct 2-cycles from \( U \setminus S' \) to \( W \setminus S' \) in \( D \setminus (S \cup S') \).
\begin{claim}\label{contra}
There is no pair of vertices  \( y \in U\setminus S' \) and  $x\in W\setminus S'$ such that there are $k+1$ internally disjoint paths from $ x$ to $y$ in $T'$ and $xyx$ is a 2-cycle in $D'$.
\end{claim}
\begin{proof}
 Suppose such a pair $x,y$ does exist. Then,  applying Lemma \ref{keylemma} to $T',x,y$ to obtain a \((k-1)\)-strong spanning oriented subdigraph  \( T'_{\{\stackrel{\leftarrow}{xy}\}} \) of \( D' \).  Combining this with the fact that \( D' \) does not contain a spanning \( k \)-strong oriented subdigraph, it follows that every \((k-1)\)-separating set of \( T'_{\{\stackrel{\leftarrow}{xy}\}} \) is also a separating set of \( T \). Therefore, either the number of \((k-1)\)-separating sets in \( T'_{\{\stackrel{\leftarrow}{xy}\}} \) is less than that in \( T' \), or \( T'_{\{\stackrel{\leftarrow}{xy}\}} \) has the same number of \((k-1)\)-separating sets as \( T' \) but the number of strong components of 
 \( T'_{\{\stackrel{\leftarrow}{xy}\}} - S \) is less than that of \( T' - S \), which contradicts the choice of \( T' \). The contradiction verifies the conclusion of this claim.
\end{proof}

To find the vertex pair $x,y$ stated in Claim \ref{contra} and further derive a contradiction, the following claim plays an essential role in the subsequent proof.
\begin{claim}\label{cycle-claim}
Let $\mathcal{C}=\{s_it_is_i\mid i\in [2k+2]\}$ be a family of distinct $2$-cycles in $D$ between $U\setminus S'$ and $W\setminus S'$, where all vertices $s_i\in U\setminus S'$ are pairwise distinct and each vertex $t_i\in W\setminus S'$ is contained in at most two such $2$-cycles. Suppose that, for any partite set $S_j$ with $|S_j\cap \{s_i\mid i\in [2k+2]\}|\geq 2$, we have $(S_j^1\cup S_j^2)\setminus S\subseteq U\setminus S'$. Then there exists a vertex $s_i$ such that $d^-_{T'[\{s_i\mid i\in [2k+2]\}]}(s_i)\geq k$.  
\end{claim}
\begin{proof}
For each $j\in [h]$, let \( l_j \) be the number of vertices from \( \{s_i\}_{i=1}^{2k+2} \) belonging to the same set \( S_j \). We first show that \( l_j \leq 2 \) for all $j\in [h]$. Suppose that \( l_j \geq 3 \) for some $j$, say without loss of generality $s_1,s_2,s_3\in S_j$. By our assumption,  \( (S^1_j\cup S^2_j) \setminus S \subseteq  U \setminus (S'\cup S_i) \). Also, there exist two distinct vertices in $\{t_1,t_2,t_3\}$, say $t_1,t_2$, such that $ S_j\leftrightarrow t_1 $ and $ S_j\leftrightarrow  t_2 $ in $D$. By rules (a)-(b), all arcs between $S_j$ and $t_1$ in $D'$ satisfy either $S_j^2\mapsto  t_1\mapsto S_j^1$ or $S_j^1\mapsto t_1\mapsto  S_j^{2}$. These arcs are single arcs in $D'$, and hence lie in $T'$. Both cases produce an arc from \( S_j \subseteq U \) to \( t_1 \in W \) in $T'$, contradicting the acyclic ordering of \( T' \). This contradiction proves \( l_j\leq 2 \) for all $j\in [h]$.

Since every good large set contains at most two vertices from \( \{s_i\}_{i=1}^{2k+2} \), each vertex in the induced subgraph $T'[\{s_i\mid i\in[2k+2]\}]$ has at most one  non-neighbor.  Now we can obtain a tournament $T''$ from \( T'[\{s_i\mid i\in [2k+2]\}] \) by adding at most one \textbf{auxiliary arc} between nonadjacent pairs of vertices in \( T'[\{s_i\mid i\in [2k+2]\}] \) (note that these auxiliary arcs are only added between vertices in the same good large set). The tournament $T''$ has $2k+2$ vertices so it has a vertex, assume \( s_1 \), satisfying $d^-_{T''}(s_1)\geq k+1$. At most one in-arc of $s_1$ is an auxiliary arc, thus $d^-_{T'[\{s_i\mid i\in [2k+2]\}]}(s_1)\geq k $, as desired. 
\end{proof}

The rest of the proof is divided into two cases according to the distribution properties of good large sets, and both cases will derive contradictions against Claim \ref{contra}.

\smallskip

\noindent \textbf{Case 1. There exists a good large set $S_i$ and $p\in [2]$ such that  $S_i^p\cap U\neq\emptyset$ and $S_i^{3-p}\cap W\neq\emptyset$}.

\smallskip

Recall that \(C(S_i)=\{v\in EX(S_i)\mid  S_i\leftrightarrow v\) in $D\}$ for each large set $S_i$.
\begin{claim}\label{claim4}
No other good large set $S_j$ with $j\neq i$ satisfies either of the following two conditions:
\begin{itemize}
\item[$(i)$] There exist vertices in $S_i$ and $S_j$ forming a $2$-cycle in $D$;
\item[$(ii)$] $S_j\cap U\neq \emptyset$ and $S_j\cap W\neq \emptyset$.
\end{itemize}
\end{claim}
\begin{proof}
Suppose there exists a good large set $S_j$ satisfying condition (i), which means $S_i\leftrightarrow S_j$ in $D$. Pick $s_i^p\in S_i^p\cap U$, $s_i^{3-p}\in S_i^{3-p}\cap W$, $s_j^p\in S_j^p$, and $s_j^{3-p}\in S_j^{3-p}$. Now rule (a) implies that $D'$ contains a  4-cycle $s_i^p\mapsto s_j^p\mapsto s_i^{3-p}\mapsto s_j^{3-p}\mapsto s_i^p$ or $s_j^p\mapsto s_i^p\mapsto s_j^{3-p}\mapsto s_i^{3-p}\mapsto s_j^p$. This  implies that $D'$ has a single arc from $U$ to $W$, which further produce arcs directed from $U$ to $W$ in $T'$, a contradiction.

{Assume now that (ii) holds for some good large set $S_j$.} According to (i), there must be $S_i\mapsto S_j$ or $ S_j\mapsto S_i $ in $D$. Assume $S_i\mapsto S_j$ in $D$ (the case that $ S_j\mapsto S_i $ is similar). Fix a vertex $s^p_j\in S_j^p\cap W$. By property (c), $D'$ contains the single arc $s^p_i\mapsto s^p_j $  from $U$ to $W$. This arc is also contained in $T'$, contradicting that there is no arc from $U$ to $W$ in $T'$. This completes the proof of Claim \ref{claim4}.
\end{proof}

It follows from Claim \ref{claim4} (i)-(ii) that 
\begin{equation}\label{e3.1}
\begin{aligned}
    \text {forming a  2-cycle with some vertex of } S_i \text { in } D \text { are contained in } C (S_i).
\end{aligned}
\end{equation}
\begin{claim}\label{claim5}
$ |C(S_i)\setminus S|\leq |S| - |S'| + 1$.
\end{claim}
\begin{proof}
 The proof proceeds by contradiction. Assume $|C(S_i)\setminus S|\geq |S|-|S'|+2$, this quantity is at least 2 by \eqref{eq1}. For each vertex $v\in C(S_i)\setminus S$, there exist $2$-cycles between $S_i$ and $v$ in $D$. In view of property (b) and the acyclic ordering of $T'$, this implies that $S_{i}^{3-p}\mapsto v\mapsto S^p_{i}$ in $D'$.  By definition, $ C(S_i)\subseteq EX(S_i)$. Hence, property (b) guarantees at least ${|C(S_i)\setminus S|-1\geq}|S|-|S'|+1$  vertices $w\in C(S_i)$ satisfying $S_i^{p}\mapsto w\mapsto S_i^{3-p}$ are single arcs in $D'$ (note that the cardinalities of these two kinds of vertex sets differ by at most one). Thus these single arcs lie in $A(T')$. Recall that \( S_i^{p} \cap U \neq \emptyset \), \( S_i^{3-p} \cap W \neq \emptyset \) and $W\mapsto U $ in $T'$.  So such vertices \( w \) must lie in \( S \) because of the acyclic ordering of $T'$. However, \( S \) already contains at least \( |S'| \) vertices from the bad large sets, which yields \( |S|\geq |S'|+|S| - |S'| + 1= |S| + 1\), a contradiction.
\end{proof}

We now prove Claim \ref{case2.1}, whose conclusion directly contradicts Claim \ref{contra}, thereby completing the proof of Case 1.

\begin{claim}\label{case2.1}
There exists a pair of vertices  \( y \in U\setminus S' \) and  $x\in W\setminus S'$ such that there are $k+1$ internally disjoint paths from $ x$ to $y$ in $T'$ and $xyx$ is a 2-cycle in $D'$.
\end{claim}

\begin{proof} {The proof sketch is as follows: By the $(4k+1)$-strong connectivity of $D$, enough disjoint paths avoiding $S\cup S'\cup C(S_i)$ yield a family of $2$-cycles $\mathcal{C}$ between $U\setminus(S'\cup S_i)$ and $W\setminus(S'\cup S_i)$. Claim \ref{cycle-claim} guarantees a vertex $y$ of sufficiently large in-degree within $U$, where $x$ is its counterpart vertex in the corresponding $2$-cycle. This in-degree condition guarantees $k+1$ disjoint $(x,y)$-paths in $T'$. Finally, verification that this $2$-cycle is preserved in $D'$ establishes the desired vertex pair.}

Since \( D \) is \(( 4k+1)\)-strong, by Claim \ref{claim5}, the number of internally disjoint paths from \( S_i\cap U \) to \( W\setminus S' \) in \( D\setminus (S\cup S'\cup C(S_i)) \) is at least
\begin{eqnarray}
    4k+1 - |S\cup S'\cup C(S_i)|&\geq& 4k+1-(|S|+|S'|+|C(S_i)\setminus{}S|\nonumber)\\
    &\geq&4k+1-(2|S|+1)\nonumber\\
    &\geq& 4k+1 - (2k-1)\nonumber\\
    &=& 2k+2.\nonumber
\end{eqnarray}
Observe that these \(2k+2\) internally disjoint paths avoid every vertex in \(C(S_i)\). Hence,  no such path in \(D\setminus (S\cup S'\cup C(S_i))\) contains an arc \(p_1p_2\) with \(p_1\in S_i\cap U\) and \(p_2\in W\setminus S'\) (since any such arc would belong to a 2-cycle of $D'$, i.e., $ p_2\in C(S_i)$ by (\ref{e3.1})).   This implies that \( |U\setminus  (S_i\cup S')|\geq 2k+2\), so using that \( |W|\geq |U| \), we obtain \( |W|\geq 2k+2 \), and then \( |W\setminus S'|\geq k+3 \) (since \( |S'|\leq |S|\leq k-1 \)). 

Select a subset $W'=\{w_1,w_2,\dots,w_{k+1}\}\subseteq W\setminus S'$ of cardinality $k+1$, and fix a vertex $u\in S_i\cap U$.  Applying Theorem \ref{menger} with \( x_1:=u \), \( y_i:=w_i \), \( a_1:=2k+2 \) and \(b_i:=2 \) for all \( i\in [k+1] \), yields $2k+2$ internally disjoint paths from $u$ to $W'$ in $D\setminus (S\cup S' \cup C(S_i))$ such that each vertex $w_i$ is the terminal vertex  of these paths. Analogously, no such path contains an arc from $S_i\cap U$ to $ W\setminus S'$. Consequently, we obtain $2k+2$ distinct 2-cycles in $D$ between $U \setminus (S'\cup S_i)$ and $W \setminus S'$ (even $W \setminus (S'\cup S_i)$ by \eqref{e3.1}). Denote these cycles by $\mathcal{C}:=\{s_it_is_i\mid i\in [2k+2]\}$, where $s_i\in U\setminus(S'\cup S_i)$ and $t_i\in W\setminus(S'\cup S_i)$. The vertices $s_i$ are pairwise distinct, and at most two 2-cycles intersect a same vertex $t_i$. {Claim \ref{claim4} (ii) yields that  $(S_j^1\cup S_j^2)\setminus S\subseteq U\setminus S'$, for any partite set $S_j$ with $|S_j\cap \{s_i\mid i\in [2k+2]\}|\geq 2$. Claim \ref{cycle-claim} gives a vertex, say $s_1$, such that $d^-_{T'[\{s_i\mid i\in [2k+2]\}]}(s_1)\geq k $.} Assume $t_1\in S_j$ (recall that $t_1\in W\setminus (S_i\cup S')$, so $ S_j\neq S_i$), thus $S_j\cap W\neq \emptyset$. {Now Claim \ref{claim4} (ii) implies that $S_j\cap U=\emptyset$}. 
Consequently,  $ |N^-_{T'[\{s_i\mid i\in [2k+2]\}]}(s_1)\cap S_j|=0$, implying that there are $k$ internally disjoint $2$-paths from $t_1$ to $y$ within $T'$.  Together with the arc $t_1 s_1$ in $T'$, we obtain $k+1$ internally disjoint paths from $x:=t_1$ to $y:=s_1$ in $T'$.

It remains to verify that the 2-cycle $xyx$ in $D$ is also a 2-cycle in $D'$. Suppose this fails. Recall that $y\in U \setminus (S'\cup S_i)$ and $x\in W \setminus (S'\cup S_i)$. By property (P3), at least one of $x\notin \bigcup_{t}S_t^3$ or $y\notin \bigcup_{t}S_t^3$ holds. Assume $ x\notin \bigcup_{t}S_t^3$. Since $ x\notin S'$, it follows that $ x\in S_j^1\cup S_j^2$ for some good large set $S_j$ ($j\neq i$). Hence Claim \ref{claim4} (ii) guarantees that $ (S_j^1\cup S_j^2)\setminus S\subseteq W$. Let $S_k$ with $k\not\in \{i,j\}$ be the set containing $y=s_1$. Then in $D$ every vertex of $S_j$  induces a 2-cycle with every vertex of $S_k$. According to (b) if $S_k$ is a large set or property (c) if $S_k$ is a singleton set, there is an arc from $U$ to $W$ in $T'$, a contradiction. 
We thus conclude that $xyx$ is a $2$-cycle in $D'$, finishing the proof of this claim.
\end{proof}

\smallskip

\noindent \textbf{Case 2.} For every good large set \( S_i \), either \( (S_i^1 \cup S_i^2) \setminus S \subseteq W \) or \( (S_i^1 \cup S_i^2) \setminus S \subseteq U \).

\smallskip

\begin{claim}\label{2-cycle}
    For a vertex $x\in W \setminus S'$ and a vertex $y\in U \setminus S'$, if $xyx$ forms a 2-cycle in $D$, then $xyx$ is also a 2-cycle in $D'$.
\end{claim}
\begin{proof}
   The claim is obvious whenever \(x,y\in \bigcup_t S_t^3\) by (P1). We may therefore suppose \(x\notin \bigcup_t S_t^3\); the case \(y\notin \bigcup_t S_t^3\) follows by symmetry. As \(x\notin S'\), we have \(x\in S_j^1\cup S_j^2\) for some good large set \(S_j\), and in this case,  \((S_j^1\cup S_j^2)\setminus S\subseteq W\). Now let \(y\in S_m\). Since \(y\in U\setminus S'\), two possibilities arise: either \(y\in C(S_j)\), or \(S_m\) is a good large set. If \(S_m\) is a good large set, then \((S_m^1\cup S_m^2)\setminus S\subseteq U\). This yields an arc from \(y\in U\) to \((S_j^1\cup S_j^2)\setminus S\subseteq W\) in \(T'\) according to rule (a), a contradiction. Next suppose \(y\in C(S_j)\) and \(|C(S_j)|\geq 2\). By rule (b), we again obtain an arc from \(y\in U\) to \((S_j^1\cup S_j^2)\setminus S\subseteq W\) within \(T'\), another contradiction. The only remaining possibility is that \(y\in C(S_j)\) and \(|C(S_j)|=1\). In this case, rule (b) guarantees that \(xyx\) forms a 2-cycle in \(D'\).
\end{proof}

\begin{claim}\label{tournament}
There is no good large set \( S_i \) such that \( S_i \cap U \neq \emptyset \), consequently, \( U \setminus S'\subseteq \bigcup_{t\in [h]} S_t^3 \).
\end{claim}
\begin{proof}
Suppose for contradiction there exists a good large set \(S_i\) such that \(S_i \cap U \neq \emptyset\); this forces \((S_i^1 \cup S_i^2)\setminus S \subseteq U\) in this case. Since \(D\) is \((4k+1)\)-strong, there exist at least \(2k+2\) internally disjoint paths from a vertex \(u\in S_i\cap U\) to \(W\) in \(D\setminus(S\cup S')\). These yield \(2k+2\) distinct 2-cycles $\mathcal{C}$ connecting vertices in \(U\setminus S'\) to vertices in \(W\setminus S'\). We first show that  $|V(\mathcal{C})\cap (U\cap S_i)|\leq 1$.\ Suppose there exist two distinct 2-cycles $u'v'u'$ and $u''v''u''$ in $\mathcal{C}$ such that $u',u''\in S_i\cap U$. Recalling that $(S_i^1\cup S_i^2)\setminus S\subseteq U$, it follows from properties (a) and (b) that either $S_i^1\mapsto v'$ or $S_i^2\mapsto v'$ in $D'$, yielding an arc from $U$ to $W$ in $T$, a contradiction. Consequently, \(|U\setminus(S_i\cup S')| \geq 2k+1\) (i.e., $|U\setminus  S'| \geq 2k+2 $). By our assumption \(|W| \geq |U|\), we have \(|W| \geq 2k+2\), so \(|W\setminus S'| \geq k+3\) as \(|S'| \leq |S| \leq k-1\). 

Select \( k+1 \) vertices \( W' := \{w_1, \dots, w_{k+1}\} \subseteq W \setminus S' \) and \( 2k+2 \) vertices \( U' := \{u_1, \dots, u_{2k+2}\} \subseteq U \setminus S' \). Applying Theorem \ref{menger}  with \( x_i = u_i \), \( x_{k+1+i} = u_{k+1+i} \), \( y_i = w_i \), \( a_i = a_{k+1+i} = 1 \), \( b_i = 2 \) for all \( i \in [k+1] \), we obtain \( 2k+2 \) internally disjoint paths from \( U' \) to \( W' \) {in $D$}. These paths yield \( 2k+2 \) distinct 2-cycles \(\mathcal{C}:= \{s_i t_i s_i \mid i \in [2k+2]\} \) between \( U \setminus S' \) and \( W \setminus S' \), where \( s_i \in U\setminus S' \) and \( t_i \in W\setminus S' \). The vertices $s_i$ are pairwise distinct, and at most two 2-cycles intersect a same vertex $t_i$. {Under the setting of Case 2, the inclusion $(S_j^1\cup S_j^2)\setminus S\subseteq U\setminus S'$ clearly holds for every partite set $S_j$ satisfying $|S_j\cap \{s_i\mid i\in [2k+2]\}|\geq 2$. By Lemma \ref{cycle-claim}, there exists a vertex $s_1$ such that $d^-_{T'[\{s_i\mid i\in[2k+2]\}]}(s_1)\geq k.$} Assume $t_1\in S_j$ (recall that $t_1\notin S_i$, so $ S_j\neq S_i$), thus $S_j\cap U\neq \emptyset$. If $ |N^-_{T'[\{s_i\mid i\in [2k+2]\}]}(s_1)\cap S_j|\geq 1$, which implies that $S_j$ is a good large set with $S_j\cap W\neq \emptyset$. This is impossible, consequently,  $ |N^-_{T'[\{s_i\mid i\in [2k+2]\}]}(s_1)\cap S_j|=0$, so there are $k$ internally disjoint 2-paths from $ t_1$ to $y$ in $T'$, together with Claim \ref{2-cycle}, which contradicts with Claim \ref{contra}.
\end{proof}

We complete the proof by considering two cases according to whether or not we have \( |U \setminus S'| > 2k+1 \). {Our main idea is as follows. If $|U\setminus S'|>2k+1$, there exist vertices of sufficiently large in-degree in the induced tournament subdigraph, which yields vertex pairs violating Claim \ref{contra}. If $|U\setminus S'|\leq 2k+1$, precise degree bounds lead to the same contradiction. Consequently, our initial assumption that $D'$ admits no spanning $k$-strong oriented subdigraph is invalid.}

Suppose first that \( |U \setminus S'| > 2k+1 \). Since \( D \) is \( (4k+1) \)-strong, Menger's theorem guarantees the existence of at least \( 2k+2 \) internally disjoint paths $\mathcal{P}$ from \( U \setminus S' \) to \( W \setminus S' \) in \( D \setminus (S \cup S') \) with distinct initial vertices. Choose $\mathcal{P}$ such that $|V(\mathcal{P})|$ is minimized. The minimality of $\mathcal{P}$ means that $\mathcal{P}$ {corresponds to a collection  of $2k+2$ distinct 2-cycles $\mathcal{C}$ of $D$} with distinct initial vertices. 
Let \( U':=V(\mathcal{C})\cap U (\subseteq U \setminus S') \), then \( |U'| = 2k+2 \). It follows from  Claim \ref{tournament} that $T'[U']$ is a tournament. This yields a vertex \( y \in U' \) such that \( d_{T'[U']}^-(y) \geq k+1 \), and a vertex \( x \in N_{T'[W\setminus S']}^+(y) \) in $D$.  Also we obtain $k+1$ internally disjoint $(x,y)$-paths in $T'$ and $xyx$ is a 2-cycle in \( D' \) ({by} Claim  \ref{2-cycle}). However this  contradicts Claim \ref{contra}.

Hence we must have \( |U\setminus S'| \leq 2k+1 \). By Claim  \ref{tournament},  $T'[U\setminus S']$ is a tournament, so it has  a vertex $y$ such that \( d_{T'[U\setminus S']}^-(y) \geq \frac{|U\setminus S'|-1}{2} \). By a simple calculation, in $D$, 
\[ 
\begin{aligned}
    d_{W\setminus S'}^+(y)& \geq d_{D}^+(y)-d^+_{S}(y)-d^+_{S'}(y)-d^+_{U\setminus S'}(y)\\
    &\geq {4k+1} - |S|-| S'| - {(|U\setminus S'|-1)}\\
    &\stackrel{(\ref{eq1})}{\geq } 4k+1-(k-1)-(k-1)-|U\setminus S'|{+1}\\
    &=2k+{4}-|U\setminus S'|\\
    &= 2(k-\frac{|U\setminus S'|-1}{2})+{3}\geq {3}.
\end{aligned}
\]  
Suppose  first that there exists an out-neighbour \( x\) of $y$ in $ W\setminus S' $  in $D$, so that $x\in S_i^1\cup S_i^2$ for some $i\in [h]$. Then the definition of $S'$ in (\ref{S'def}) implies that $S_i$ is good and, in Case 2.2, we must have $(S_i^1\cup S^2_i)\setminus S \subseteq W$. 
Since $x\notin S'$ it follows from Claim \ref{2-cycle} that $xyx$ is also a 2-cycle in $ D'$. Now, our aim is to lower bound the out-degree of $ x$ in $W\setminus S' $ within $T'$. Recall that $I_B := \{i\in [h] \mid S_i  \text{ is a bad large set}\}$ and the definition of $S'$ from (\ref{S'def}), we get that, 
in $D$, \[N^+_D(x)\setminus (N^+_{U\setminus (S'\cup \{y\})}(x)\cup N^+_{\bigcup_{i \in I_B}  (S_i^1 \cup S_i^2)}(x))\subseteq (W\setminus S')\cup (S\setminus \bigcup_{i \in I_B}  (S_i^1 \cup S_i^2) )\cup \{y\}. \] Define $K:=(W\setminus S')\cup (S\setminus \bigcup_{i \in I_B}  (S_i^1 \cup S_i^2) )\cup \{y\}$. Thus, in $D$,
$$ d^+_K(x)\geq d^+_D(x)-d^+_{U\setminus (S'\cup \{y\})}(x)-d^+_{\bigcup_{i \in I_B}  (S_i^1 \cup S_i^2)}(x). $$ 
Additionally, each arc in $D$ from $x$ to a vertex $v\in N^+_K(x)$ has one of the following properties:
\begin{itemize}
    \item the arc $e$ is a single arc in $D$, let $E_1$ be the set of such single arcs, and let $m_1$ be the size of $E_1$;
    \item the vertex $v$ belongs to $C(S_i)$;  
    \item the vertex $v$ belongs to a good large set $S_j$, in this case, since $S_j$ is good large set, it follows from Claim \ref{tournament} that we have $S_j\cap ( \bigcup_{i \in I_B}  (S_i^1 \cup S_i^2)\cup U \setminus (S'\cup \{y\})) =\emptyset$. 
\end{itemize}
Because of property (b) and the acyclic order of $T'$, we must have $ S_i\mapsto U\setminus (S'\cup \{y\})$ in $D$.  Consequently, $ N^+_K(x)\setminus V(E_1)= C(S_i)\cup \bigcup_{t\in I}V(S_t)$ in $D$,   where $I\subseteq \{1,\ldots{},h\}$ is the set of indices of those good large sets that form $2$-cycles with $S_i$ in $D$. By (P2), the number of single out-arcs of $x$ in $K$ within $D'$ is at least 
$$ 
\begin{aligned}
    & ( d^+_D(x)-d^+_{U\setminus (S'\cup \{y\})}(x)-d^+_{\bigcup_{i \in I_B}  (S_i^1 \cup S_i^2)}(x) -m_1)/2-1+m_1\\
    &\geq ( d^+_D(x)-d^+_{U\setminus (S'\cup \{y\})}(x)-d^+_{\bigcup_{i \in I_B}  (S_i^1 \cup S_i^2)}(x))/2-1.
\end{aligned}
$$ 
Further, in $T'$, 
\[
\begin{aligned}
    d^+_{T'[W\setminus S']}(x)&\stackrel{(P2)}{\geq } (d^+_D(x)-d^+_{U\setminus (S'\cup \{y\})}(x)-d^+_{\bigcup_{i \in I_B}  (S_i^1 \cup S_i^2)}(x))/2-1-|S\setminus \bigcup_{i \in I_B}  (S_i^1 \cup S_i^2)|-|\{y\}|\\
    &\geq \left(4k+1-(|U\setminus S'|-1)-|\bigcup_{i \in I_B}  (S_i^1 \cup S_i^2)|\right)/2-2-|S\setminus \bigcup_{i \in I_B}  (S_i^1 \cup S_i^2)|\\
    &\stackrel{(\ref{s'})}{\geq } \left(4k-|U\setminus S'|-2|S\cap \bigcup_{i \in I_B}  (S_i^1 \cup S_i^2)|\right)/2-1-|S|+|S\cap \bigcup_{i \in I_B}  (S_i^1 \cup S_i^2)|  \\
    &=\left(4k+4-|U\setminus S'|\right)/2-2-|S|\geq k-|U\setminus S'|/2.
\end{aligned}
\]
Thereby,
    $$d^+_{T'[W\setminus S']}(x)+d_{T'[{U}\setminus S']}^-(y)\geq \lceil k-|U\setminus S'|/2+\frac{|U\setminus S'|-1}{2}\rceil \geq k.$$ 
Together with the arc $ xy$ in $T$, we can find $k+1$ internally disjoint paths from $x$ to $y$ in $T'$, which contradicts Claim \ref{contra}. 

Thus every out-neighbour of $y$ in $W\setminus S'$ belong to $EX(y)$.
Hence, $y$ has $ 2(k-\frac{|U\setminus S'|-1}{2})+1$ out-neighbours in $ EX(y)\cap (W\setminus S') $  within $D$. Such $ 2(k-\frac{|U\setminus S'|-1}{2})+1$ out-neighbours induce a tournament in $T'$, so there is a vertex $x$ with $d^+_{T'[W\setminus S']}(x)\geq k-\frac{|U\setminus S'|-1}{2} $ in $T'$.  This implies  that 
    $$d^+_{T'[W\setminus S']}(x)+d_{T'[U \setminus S']}^{-}(y)\geq k-\frac{|U\setminus S'|-1}{2}+\frac{|U\setminus S'|-1}{2} \geq k,$$ 
    which contradicts Claim \ref{contra}. Hence the proof is completed.
\end{proof}

\section{Proof of Theorem \ref{main1}}




We now complete the proof of Theorem \ref{main1} using Theorem \ref{extended}. For convenience we repeat the statement of  Theorem \ref{main1}. 

\noindent{}{\bf Theorem \ref{main1}.}
 For every integer $k \geq 1$, every $(4k+1)$-strong semicomplete composition  $D = H[S_1, \dots, S_h]$, {satisfying that either $h\geq 3$ or }$|V(D) \setminus S_i| \geq 2k$ for all $S_i$, contains a spanning $k$-strong oriented subgraph.

\begin{proof}
The proof is divided into two cases depending on whether $h=2$ or not. If $h = 2$, for each part $S_i$ of the composition $D$, by the assumption $\left| V(D) \setminus S_i \right| \geq 2k$ for all $i\in [h]$, we know that $|S_1| \geq 2k$ and $|S_2| \geq 2k$. We may partition each $S_i$ into two disjoint non-empty subsets $S_i^1$ and $S_i^2$ that are as balanced as possible, satisfying $|S^1_i| = \lfloor |S_i| \rfloor /2 $, $ |S^2_i| = \lceil |S_i| \rceil /2$ and $S_i^1 \cap S_i^2 = \emptyset$. Consequently, $|S_i^j| \geq k$ for $i,j \in [2]$. Let $D'$ be the digraph with vertex set $V(D') = V(D)$ and arc set satisfying $S^1_1 \mapsto  S^1_2 \mapsto S^2_1\mapsto S^2_2 \mapsto S^1_1$, and that $D'$ has no arc inside $S_1$ or $S_2$. It is easy to verify that $D'$ is the desired $k$-strong oriented subgraph. If $h \geq 3$, let $D_0 = H[S_1^0, \dots, S_h^0]$ be the digraph obtained from $D$ by removing all the arcs that lie inside $S_i$. 
It follows from Lemma \ref{scd-com} that $D_0$ is a $(4k+1)$-strong extended semicomplete digraph. By Theorem \ref{extended}, $D_0$ contains a spanning $k$-strong oriented subdigraph, as desired, which completes the proof.
\end{proof}

\section{Proof of Theorem \ref{thm22}}

For convenience we repeat the statement of the theorem here.

\noindent{}{\bf Theorem \ref{thm22}.} 
Let $k$ be a positive integer. Every $5k$-strong semicomplete split digraph $D=(V_1,V_2,A)$ contains a spanning \( k \)-strong oriented semicomplete split digraph.

We prove the theorem by induction on $k$.
In the base case $k=1$, we are dealing with  a 5-strong split digraph on at least 5 vertices. Clearly such a digraph has no vertex partition $(X,V\setminus X)$ with precisely two arcs between the sets. Hence the claim follows from {Theorem \ref{thm:mixedkarcstrong}}. 
Now consider the case $k \geq 2$. Assume that the statement holds for all $k -1$. 
By the induction hypothesis for $k-1\ge 1$, $D$ admits a spanning $(k-1)$-strong oriented semicomplete split digraph $T'$. Assume for contradiction that \( D \) does not have any spanning \( k \)-strong oriented semicomplete split digraph. Let $T$ be a spanning $(k-1)$-strong oriented semicomplete split subdigraph of $D$ satisfying the following minimality conditions:
\begin{enumerate}[label=(\roman*), leftmargin=*]
    \item[(i)] The number of minimum separating sets of $T$ is as small as possible.
    \item[(ii)] $T$ has a minimum separating set $S$ such that the number of strong components of $T-S$ is as small as possible.
\end{enumerate}
Let \( S \) be such a \((k-1)\)-separating set of \( T \),  and let \( p \) be the number of strong components  in \( T - S \), which is minimal over all choices of \( S \). Let \( T_1, T_2, \ldots, T_p \) denote an acyclic order of \( T - S \) (i.e., for any \( i < j \), \( T_i\mapsto T_j \) in $T$). Let \( W = V(T_1 \cup \cdots \cup T_{p-1}) \) and \( U = V(T_p) \). We assume \( |W| \geq |U| \) (the case \( |U| > |W| \) can be handled by considering the converse digraph of \( D \)). The rest of the proof splits into two cases depending on the size of $U$.

\noindent \textbf{Case 1.  $|U|\geq 4k$}

By the assumption above, we also have  $|W|\geq 4k$. By Menger's theorem, there are $4k$ disjoint paths from $U$ to $W$ in $T-S$. They correspond to $4k$ disjoint 2-cycles in $D$. Let $\mathcal{C}$ be the set of these 2-cycles, let $U':=U\cap V(\mathcal{C})$, and let $W':=W\cap V(\mathcal{C})$. By the pigeonhole principle, there are at least $2k$ vertices in $W'$ that are in $V_2$ (semicomplete) or in $V_1$. If they are in $V_1$, then $U'$ must have at least $2k$ vertices in $V_2$. In short, either $|U'\cap V_2|\geq 2k$ or $|W'\cap V_2|\geq 2k$; assume without loss of generality 
that  $|U'\cap V_2|\geq 2k$. Then $U'$ contains a vertex $x\in V_2$ such that $x$ has at least $k$ in-neighbours $X^-$ in $V_2\cap U$. Assume that the 2-cycle in $\mathcal{C}$ adjacent to $x$ is $xyx$. Then $y\in W$ and $y\rightarrow X^-$ in $T$. According to Lemma \ref{keylemma}, we obtain a \((k-1)\)-strong spanning oriented subdigraph \( T_{\{\stackrel{\leftarrow}{yx}\}} \) of \( D \). Combining this with the fact that \( D \) does not contain a \( k \)-strong spanning oriented subdigraph, it follows that every \((k-1)\)-separating set of \( T_{\{\stackrel{\leftarrow}{yx}\}} \) is also a separating set of \( T \). Therefore, either the number of \((k-1)\)-separating sets in \( T_{\{\stackrel{\leftarrow}{yx}\}} \) is less than that in \( T \), or \( T_{\{\stackrel{\leftarrow}{yx}\}} \) has the same number of \((k-1)\)-separating sets as \( T \) but the number of strongly connected components of \( T_{\{\stackrel{\leftarrow}{yx}\}}-S \) is less than that of \( T - S \), which contradicts the choice of \( T \).

  \noindent \textbf{Case 2.  $|U|< 4k$}

As $D$ is $f(k)$-strong, by Menger's theorem, every vertex of $U$ is incident with a 2-cycle whose other vertex is in $W$. Suppose $m_1:=|U\cap V_1|$ and $m_2:=|U\cap V_2|$, then $m_1+m_2=|U|$. For each vertex $v\in U\cap V_1$, we have $d^+_{D[W]}(v)\geq f(k)-m_2-|S|> m_1$ in $D$ and $N^+_{D[W]}(v)\subseteq V_2$. Thus there is a vertex $v^+\in N^+_{D[W]}(v)$ with $d^+_{T[W\cap V_2]}(v^+)\geq \frac{d^+_{D[W]}(v)-1}{2}> \frac{m_1-1}{2}$ in $T$. If $\dfrac{m_1-1}{2}\geq k $, then applying Lemma \ref{keylemma}, as above (to the arc $v^+v$) we get a contradiction. Thus we must have 
   \begin{equation}\label{e1}
   \dfrac{m_1-1}{2}< k.
   \end{equation}
   As $T[U\cap V_2]$ is a tournament, there is a vertex $x\in U\cap V_2$ with $d^-_{T[U\cap V_2]}(x)\geq (m_2-1)/2$. Also, $d^+_{W}(x)\geq f(k)-|S|-(|U|-1)\geq 2$ in $D$, since $x$ has at most $|U|-1$ out-neighbours in $U$. Let $z\in N^+_{W}(x)$, be chosen so that $z\leftrightarrow x$ is a 2-cycle in $D$. Lemma \ref{keylemma} applied to the arc $zx$ implies that
   \begin{equation}\label{e3}
  (m_2-1)/2< k, \text{ i.e. } m_2\leq 2k.
   \end{equation}

If $N^+_{W}(x)\subseteq V_2$, then there is a vertex $w\in N^+_{W}(x)$ with $d^+_{T[W\cap V_2]}(w)\geq \frac{f(k)-k+1-|U|}{2}$. On the other hand,
\begin{equation}\label{e2}
  (m_2-1)/2+ \frac{f(k)-k+1-|U|}{2}\leq d^-_{T[U\cap V_2]}(x)+d^+_{T[W\cap V_2]}(w)< k.
   \end{equation}
   \noindent where the last inequality follows Lemma \ref{keylemma} as $w\rightarrow U \cap V_2$ and $W \cap V_2 \to x$.
   Using that $|U|=m_1+m_2$, and combining with (\ref{e1}) and (\ref{e2}), we obtain that $ f(k)< 5k $, a contradiction.
   Hence, $N^+_{W}(x)\cap V_1\neq \emptyset$.   Let $y\in N^+_{W}(x)\cap V_1$, and note that, in $D$,    
   \begin{equation}\label{y-outinD}
       d^+_{D[W]}(y)\geq f(k)-m_2-|S|\geq 2k+1,
\end{equation} 
where the first inequality holds because no vertex in $U\cap V_1$ is adjacent to $y$. 
   As $yxy$ is a 2-cycle in $D$, Lemma \ref{keylemma} implies that we have 
   \begin{equation}\label{boundinT}
       d^+_{T[W]}(y)+ d^-_{T[U\cap V_2]}(x)<k.
   \end{equation} Let $m_3:=d^+_{T[W]}(y)$. It follows from (\ref{y-outinD}) that there is a set $F\subseteq W\cap V_2$ of order $f(k)-m_2-|S|-m_3$ such that for each vertex $s\in F$, $sys$ is a 2-cycle in $D$, but $s\rightarrow y$ in $T$. Let $l:=k-\frac{m_2-1}{2}-m_3$. And pick $w_1,w_2,\ldots,w_l$ as the $l$ vertices in $T[F]$ with maximum out-degree. As $T[F]$ is a tournament, each vertex $w_i$ satisfies 
   \[
   \begin{aligned}
       d^+_{T[F]}(w_i)&\geq \frac{f(k)-m_2-|S|-m_3-l-1}{2}\\
       &= \frac{f(k)-m_2-|S|-m_3-(k-\frac{m_2-1}{2}-m_3)-1}{2}\\
       &\geq \frac{f(k)-\frac{m_2}{2}-k+1-k-1-\frac{1}{2}}{2}\stackrel{(\ref{e3})}{\geq } k-\frac{1}{4}.
   \end{aligned}
   \] 
   
   This implies that for each $w_i$, $i\in [l]$, there are at least $k$ internally disjoint $(w_i, y)$-paths of length 2 in $T[W]$. By iteratively applying Lemma \ref{keylemma} and reversing the $l+1$ arcs $w_1y,w_2y,\dots,w_ly,yx$ in sequence, we obtain a \((k-1)\)-strong spanning oriented subdigraph $T_{\{\stackrel{\leftarrow}{w_1y}, \stackrel{\leftarrow}{w_2y},\dots,\stackrel{\leftarrow}{w_ly},\stackrel{\leftarrow}{yx}\}}$ of $T$. Combining this with the fact that \( D \) does not contain a \( k \)-strong spanning oriented subdigraph, it follows that every \((k-1)\)-separating set of $T_{\{\stackrel{\leftarrow}{w_1y}, \stackrel{\leftarrow}{w_2y},\dots,\stackrel{\leftarrow}{w_ly},\stackrel{\leftarrow}{yx}\}}$ is also a separating set of \( T \). Therefore, either the number of \((k-1)\)-separating sets in $T_{\{\stackrel{\leftarrow}{w_1y}, \stackrel{\leftarrow}{w_2y},\dots,\stackrel{\leftarrow}{w_ly},\stackrel{\leftarrow}{yx}\}}$ is less than that in \( T \), or $T_{\{\stackrel{\leftarrow}{w_1y}, \stackrel{\leftarrow}{w_2y},\dots,\stackrel{\leftarrow}{w_ly},\stackrel{\leftarrow}{yx}\}}$ has the same number of \((k-1)\)-separating sets as \( T \) but the number of strongly connected components of $T_{\{\stackrel{\leftarrow}{w_1y}, \stackrel{\leftarrow}{w_2y},\dots,\stackrel{\leftarrow}{w_ly},\stackrel{\leftarrow}{yx}\}}$ is less than that of \( T - S \), which contradicts the optimality of \( T \). {Note that this case is slightly more complicated as it requires reversing multiple arcs.} This completes the proof.  \hfill{}$\Box$

\bibliographystyle{plain} 
\bibliography{ref}
\end{document}